%% file: paper.tex
\documentclass[english]{ourlematema}
\usepackage{amsmath, amsthm, amssymb, hyperref, color, bm, enumitem, caption}
\usepackage{MnSymbol} 
\usepackage{graphicx}
\usepackage{caption}
\usepackage[all]{xypic}
\usepackage{verbatim}
\usepackage{chemfig, chemnum}
\usepackage{tikz}
\usetikzlibrary{calc}
\usetikzlibrary{math}
\usepackage[linesnumbered,lined,commentsnumbered]{algorithm2e}
\usepackage{caption}
\usepackage{subcaption}
\usepackage{wrapfig}
\usepackage{float}
\usepackage{makecell}
\usepackage{array}
\usepackage{fancyvrb}

\newtheorem{theorem}{Theorem}
\numberwithin{theorem}{section}
\newtheorem{proposition}[theorem]{Proposition}

\newtheorem{corollary}[theorem]{Corollary}
\newtheorem{definition}[theorem]{Definition}

\newtheorem{example}[theorem]{Example}
\newtheorem{conjecture}[theorem]{Conjecture}

\theoremstyle{definition}

\newcommand{\PP}{\mathbb{P}}
\newcommand{\bP}{\mathbb{P}}
\newcommand{\RR}{\mathbb{R}}
\newcommand{\bR}{\mathbb{R}}

\newcommand{\bC}{\mathbb{C}}

\newcommand{\bZ}{\mathbb{Z}}

\newcommand {\cT} {\mathcal T}
\newcommand{\cM}{\mathcal{M}}
\newcommand{\rA}{\mathrm{A}}
\newcommand{\rC}{\mathrm{C}}
\newcommand{\fS}{\mathfrak{S}}
\newcommand {\wt} {\widetilde}

\newcommand {\bs} {\backslash}
\newcommand{\AS}{\mathrm{as}}
\newcommand{\CS}{\mathrm{cs}}
\newcommand {\pp} {{+}{+}}

\DeclareMathOperator{\Trop}{Trop}
\DeclareMathOperator{\init}{in}

\DeclareMathOperator{\sgn}{sgn}
\DeclareMathOperator{\Gr}{Gr}
\DeclareMathOperator{\SR}{SR}

\mathchardef\standardl=\mathcode`l
\mathcode`l=\ell
\newcommand{\deactivatel}{\mathcode`l=\standardl}
\makeatletter
\edef\operator@font{\operator@font\noexpand\deactivatel}
\makeatother

\title{Tropicalizing binary geometries}
 
\author{Shelby Cox}
\address{%
  MPI for Mathematics in the Sciences, Leipzig, Germany \\
\email{shelby.cox@mis.mpg.de}
}

  \author{Igor Makhlin}
  \address{Technische Universität Berlin, Berlin, Germany\\
\email{iymakhlin@gmail.com}
}

\date{2024/10/15}

\begin{document}
\maketitle
\begin{abstract}
\noindent 
The type A cluster configuration space, commonly known as $\cM_{0,n}$, is the very affine part of the binary geometry associated with the associahedron. 
The tropicalization of $\cM_{0,n}$ can be realized as the space of phylogenetic trees and its signed tropicalizations as the dual-associahedron subfans. 
We give a concise overview of this construction and propose an extension to type C. 
The type C cluster configuration space $\cM_{\rC_l}$ arises from the binary geometry associated with the cyclohedron. 
We define a space of axially symmetric phylogenetic trees containing many dual-associahedron and dual-cyclohedron subfans. 
We conjecturally realize the tropicalization of $\cM_{\rC_l}$ as the defined space and its signed tropicalizations as the aforementioned subfans.
\end{abstract}

\section*{Introduction}

The notion of binary geometries was introduced in~\cite{AHLT} in the context of particle scattering in arbitrary space-time dimension. In particle physics it arises from the ABHY kinematic associahedron (\cite{ABHY}) and its extension \cite{BDMTY} to generalized associahedra of other finite-type cluster algebras (in the sense of~\cite{FZ,CFZ}). 

From a cluster-algebraic standpoint one may view binary geometries as cluster configuration spaces, i.e.\ partial compactifications of torus quotients of cluster varieties, cf.\ also~\cite{HLPZ,AHL}. Such binary geometries are realized as affine varieties cut out by the so-termed $u$-equations: defining relations explicitly determined by the combinatorics of the generalized associahedron. In particular, the resulting variety admits a stratification governed by the face poset of the same polytope. In~\cite{L} this definition is widely generalized to the setting of flag simplicial complexes. In this paper we follow the approach in~\cite{L}, nonetheless, our main objects of study are the binary geometries given by the type A and type C associahedra, and the respective finite-type cluster configuration spaces.

Let us first consider type A. The cluster configuration space $\cM_{\rA_{n-3}}$ is commonly denoted by $\cM_{0,n}$; it is the moduli space of $n$-tuples of distinct points in $\bP^1$. It may also be viewed as the very affine part of the binary geometry $\wt\cM_{0,n}$ associated with the simplial complex dual to the associahedron. The tropicalization of $\cM_{0,n}$ (and, thus, of $\wt\cM_{0,n}$) is well-studied; its description can be found in~\cite{GKM}, the book~\cite[Section 6.4]{ITG} and other sources but traces back at least to~\cite{SpSt,T}. This description characterizes $\Trop\cM_{0,n}$ as the space of phylogenetic trees introduced in~\cite{RW,BHV}. Furthermore, $\Trop\cM_{0,n}$ is naturally realized as the union of $(n-1)!/2$ dual-associahedron subfans. These are precisely the signed tropicalizations (i.e.\ generalized positive tropicalizations) corresponding to all sign patterns of coordinates occurring in $\cM_{0,n}(\bR)$. In the first sections of this paper we recall the definitions and results pertaining to $\cM_{0,n}$ and its tropicalization. We also write down explicit formulas for the cones of $\Trop\cM_{0,n}$ in the $u$-coordinates arising from the binary geometry structure. This overview provides a background and motivation for our main results which concern type~C.

Our first result (see Sections~\ref{ASsec},~\ref{CSsec}) is the construction of an adequate analogue of the space of phylogenetic trees for the symplectic setting. The previously mentioned description of $\Trop\cM_{0,n}$ realizes it as the fan over a simplicial complex with faces enumerated by phylogenetic trees with $n$ labeled leaves. In type C we, instead, consider trees with $2n$ labeled leaves that are \textbf{axially} symmetric in a natural sense. We define the\textit{ complex of axially symmetric phylogenetic trees} $\Theta_{\AS}(n)$, a simplicial complex of dimension $n-2$ with faces enumerated by such trees. It is naturally the union of $2^{n-2}n!$ dual-associahedron subcomplexes. We also consider trees with $2n$ labeled leaves that are \textbf{centrally} symmetric. The resulting \textit{complex of centrally symmetric phylogenetic trees} $\Theta_{\CS}(n)$ is shown to be a subcomplex of $\Theta_{\AS}(n)$. Furthermore, the subcomplex $\Theta_{\CS}(n)$ is itself the union of $2^{n-2}(n-1)!$ dual-cyclohedron subcomplexes.

Next, recall that the cyclohedron is the generalized associahedron of type $\rC_{n-1}$. Accordingly, the cluster configuration space $\cM_{\rC_{n-1}}$ can be realized in \textit{$u$-coordinates} as the very affine part of the binary geometry defined by the dual-cyclohedron complex, see~\cite{AHL}. 
This binary geometry is an affine variety of dimension $n-1$. Our second result is a conjectural description of its tropicalization.

\begin{conjecture}[{cf.\ Conjectures~\ref{mainconj},~\ref{signedconj}}]\label{introconj}
With respect to the $u$-coordinates, $\Trop\cM_{\rC_{n-1}}$ is combinatorially equivalent to a fan over the complex of axially symmetric phylogenetic trees $\Theta_\AS(n)$. Furthermore, the subfans given by the $2^{n-2}n!$ dual associahedra and the $2^{n-2}(n-1)!$ dual cyclohedra are the signed tropicalizations corresponding to all sign patterns of coordinates occurring in $\cM_{\rC_{n-1}}(\bR)$.
\end{conjecture}

In particular, the cones of $\Trop\cM_{\rC_{n-1}}$ correspond to faces of $\Theta_\AS(n)$. We also conjecture an explicit description in $u$-coordinates for every such cone. This proposed description should be viewed as an extension of the $u$-coordinate realization that was given for $\Trop\cM_{0,n}$. 

The main motivation for our conjectures is the case $n=3$ which is already highly nontrivial: $\Trop\cM_{\rC_2}$ is a fan in ambient dimension 6 with 21 maximal cones and 13 rays. This tropicalization can be found using Macaulay2 (\cite{M2}) and agrees fully with the hypothesized structure. In addition, the number of sign patterns provided by Conjecture~\ref{introconj} agrees with a conjecture made in~\cite{AHL}.

We hope the complex $\Theta_\AS(n)$ and our conjecture, if proved, to have several applications. For example, it is natural to expect $\Trop\cM_{\rC_{n-1}}$ to be directly related to the type C tropicalized cluster variety. Also, the well-studied Deligne--Mumford compactification $\overline\cM_{0,n}$ has a stratification governed by the face poset of $\Trop \cM_{0,n}$. One may hope for a similar construction to exist in type C. 

\section{Tropicalizations}

We give a short overview of tropicalizations and positive tropicalizations of polynomial ideals and algebraic varieties. The results concerning tropicalizations can be found in~\cite{SpSt}; for positive tropicalizations see~\cite{SpW}.

For a finite set $V$ consider the polynomial ring $S = \bC[u_i]_{i\in V}$. A real weight $w \in \bR^V$ can be viewed as the $\bR$-grading on $S$ that is equal to $w_i$ on $u_i$. A polynomial $p\in S$ can be decomposed into the sum of its $w$-homogeneous components. The \textit{initial part} $\init_w p$ is the nonzero homogeneous component of the minimal occurring grading. For an ideal $I \subset S$ we define the \textit{initial ideal} $\init_w I$ as the linear span of the set $\{\init_w p\}_{p\in I}$. It is easily checked that $\init_w I$ is also an ideal.

For a pair of ideals $I, J \subset S$ let $C(I,J)^\circ$ denote the set of all $w \in \bR^V$ for which $\init_w I = J$ and let $C(I,J)$ denote the closure of $C(I,J)^\circ$. Now suppose that the ideals $I$ and $J$ are monomial-free, i.e.\ contain no monomials. In this case the set $C(I,J)$ is a closed polyhedral cone and $C(I,J)^\circ$ is its relative interior (if both are nonempty). Furthermore, these cones form a polyhedral fan.
\begin{definition}
For a given monomial-free ideal $I\subset S$ the polyhedral fan formed by all nonempty cones of the form $C(I,J)$ with $J$ monomial-free is the \textit{tropicalization} of $I$. It is denoted by $\Trop I$.
\end{definition}

Usually $I$ arises as the defining ideal of an affine, projective or very affine variety $X$, and in this case $\Trop I$ is also referred to as $\Trop X$. However, note that the fan $\Trop X$ is not determined by $X$; it also depends on the chosen affine or projective embedding which is always implicit in the notation.

A key property of $\Trop I$ is that it is pure of dimension equal to the Krull dimension of $S/I$. Hence, $\Trop X$ has dimension $\dim X$ if $X$ is affine and $\dim X+1$ if $X$ is projective.


Next, the positive tropicalization is a subfan of the tropicalization.
\begin{definition}
For $I$ as above the \textit{positive tropicalization} $\Trop_{> 0} I$ is the fan formed by those cones $C(I,J)$ for which $J$ contains no elements of $\bR_{>0}[u_i]_{i\in V}$.
\end{definition}

In other words, $\Trop_{> 0} I$ is formed by those $C(I,J)$ for which every nonzero polynomial in $J$ has both positive and negative coefficients. In particular, let $X \subset \bC^V$ be the zero set of $I$ and let $X_{>0} = X \cap \bR_{>0}^V$ be its totally positive part. One sees that if $X_{>0}$ is nonempty, then $I\cap\bR_{>0}[u_i]_{i\in V}=0$, hence $\Trop_{> 0} I$ is nonempty. The converse is not true in general.

The above observation motivates us to generalize the notion of positive tropicalization to general sign patterns where a sign pattern is an element of $\{\pm1\}^V$.
\begin{definition}
Consider a sign pattern $\tau$ and the automorphism $\varepsilon_\tau$ of $S$ taking $u_i$ to $\tau_i u_i$. The \textit{signed tropicalization} $\Trop_\tau I$ (also $\Trop_\tau X$) is the positive tropicalization $\Trop_{> 0} \varepsilon_\tau(I)$.
\end{definition}

Note that the tropicalizations of $\varepsilon_\tau(I)$ and $I$ are the same, hence $\Trop_\tau I$ is also a subfan of $\Trop I$. Furthermore, one sees that $\Trop_\tau I$ is nonempty if the sign pattern $\tau$ \textit{occurs} in $X(\bR)=X\cap\bR^V$, which means that there exists a point $x\in X(\bR)$ such that $\sgn x_i=\tau_i$ for all $i$.


\section{Phylogenetic trees and dual associahedra}\label{phyltrees}

We overview the definition and basic properties of the complex (or space) of phylogenetic trees constructed in \cite{RW,BHV}. Choose an integer $n \ge 3$. 

\begin{definition}
A \textit{phylogenetic tree} with $n$ leaves is a tree $T$ with no vertices of degree two together with a bijection $\varphi$ from $[n]$ to the set of leaves of $T$. For $i \in [n]$ the leaf $\varphi(i)$ is said to be labeled by $i$. We identify two phylogenetic trees $(T_1,\varphi_1)$ and $(T_2,\varphi_2)$ if there is a graph isomorphism between $T_1$ and $T_2$ mapping $\varphi_1(i)$ to $\varphi_2(i)$ for all $i\in[n]$. 
\end{definition}

Phylogenetic trees naturally arise from subdivisions of polygons. First, we recall the following notion: a \textit{dihedral ordering} of $[n]$ is a permutation of $[n]$ considered up to compositions of cyclic shifts and reversals. In other words, a dihedral ordering is a way of labeling the edges of a regular $n$-gon by elements of $[n]$, considered up to the natural action of the dihedral group. It is straightforward to show that there are $(n-1)!/2$ dihedral orderings of $[n]$.

\begin{wrapfigure}{r}{3.5cm}
    \includegraphics{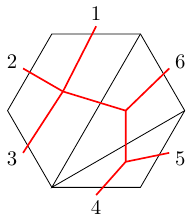}
\end{wrapfigure}
Choose a dihedral ordering $\alpha$, and consider a convex $n$-gon $P$ with edges labeled in accordance with $\alpha$. Every subdivision of $P$ defines a phylogenetic tree as follows. First consider the adjacency graph of the 2-dimensional cells of the subdivision, this is a tree. Then, for the cell $C$ containing the edge of $P$ labeled by $i\in [n]$ add a leaf adjacent to the vertex corresponding to $C$, label this leaf by $i$ (see figure on the right). The result is a phylogenetic tree with $n$ leaves. We say that a phylogenetic tree is \textit{compatible} with $\alpha$ if it arises from a subdivision of $P$ in the above way.


There are $n(n-3)/2$ coarsest nontrivial subdivisions of $P$: those formed by a single diagonal. The corresponding phylogenetic trees are those that have exactly two non-leaf vertices. Specifically, a diagonal partitions the edge labels of $P$ into two subsets $A$ and $B$ with $|A|,|B| \ge 2$. In the corresponding tree one non-leaf vertex is adjacent to the leaves labeled by elements of $A$ and the other to the leaves labeled by elements of $B$.
\begin{definition}
We define a simplicial complex $\Delta(\alpha)$ whose vertex set consists of the phylogenetic trees that are compatible with $\alpha$ and have exactly two non-leaf vertices. A subset of vertices forms a face if and only if the corresponding subdivisions of $P$ have a common refinement or, equivalently, if the corresponding diagonals are pairwise non-crossing (we say that two diagonals cross if they are distinct and share interior points).
\end{definition}

We see that the faces of $\Delta(\alpha)$ are in bijection with all subdivisions of $P$ and, therefore, with all phylogenetic trees compatible with $\alpha$.\footnote{
We consider the empty set to be a face of any simplicial complex, here it corresponds to the trivial subdivision and the phylogenetic tree with a single non-leaf vertex.}
This is a flag complex of pure dimension $n-4$ (recall that a simplicial complex is a \textit{flag complex} if all of its minimal non-faces are of size two).\footnote{
For $n=3$ the complex $\Delta(\alpha)$ is empty and thus has a single face and is of dimension $-1$.}
We refer to $\Delta(\alpha)$ as a \textit{dual associahedron} because its face poset is dual to that of an associahedron. 
For example, if $\alpha$ is any dihedral ordering of five elements, $\Delta(\alpha)$ is a 5-cycle.

We denote $\Delta(n)=\Delta(\alpha)$ for $\alpha$ the dihedral ordering given by the identity permutation. Obviously, for a fixed $n$ every $\Delta(\alpha)$ is isomorphic to $\Delta(n)$, however, we view the various $\Delta(\alpha)$ as distinct complexes because they have distinct vertex sets. It is important to note that these vertex sets are not necessarily disjoint since a phylogenetic tree can be realized in the plane in different ways, meaning that it will be compatible with multiple dihedral orderings. In particular, this allows for the following definition. 

\begin{definition}
    \label{coptdef}
    The \textit{complex of phylogenetic trees} $\Theta(n)$ is the union of $\Delta(\alpha)$ over all $\alpha$, with vertices represented by the same phylogenetic trees identified. It is the simplicial complex whose vertex set consists of all phylogenetic trees on $n$ leaves that have exactly two non-leaf vertices. A subset of vertices forms a face if and only if it forms a face in one of the $\Delta(\alpha)$.
\end{definition} 
The faces of $\Theta(n)$ are in bijection with all phylogenetic trees with $n$ leaves which explains the terminology. Evidently, $\Theta(n)$ is also a flag complex of pure dimension $n-4$. The fan over this simplicial complex is known as the \textit{space of phylogenetic trees}, a term introduced in~\cite{BHV} (occasionally $\Theta(n)$ is referred to by the same name but we distinguish between the two). For example, $\Theta(5)$ is the Petersen graph, see \cite[FIG.\ 13]{BHV}. The Petersen graph contains 12 copies of the 5-cycle, and these are precisely the 12 subcomplexes $\Delta(\alpha)$ corresponding to the 12 dihedral orderings. 


To conclude this section we explain how containment of faces of $\Theta(n)$ can be read off from the corresponding trees. For a phylogenetic tree $(T,\varphi)$ and a non-leaf edge $e$ of $T$ we may consider the phylogenetic tree $(T/e,\psi\circ\varphi)$ where $\psi$ is the natural bijection from the leaf set of $T$ to that of the contraction $T/e$. We say that $(T/e,\psi\circ\varphi)$ is obtained from $(T,\varphi)$ by contracting $e$. 

As above, consider a convex $n$-gon $P$ with edges enumerated by $[n]$. In terms of the correspondence between subdivisions of $P$ and phylogenetic trees, the contraction of a non-leaf edge simply corresponds to the deletion of the respective diagonal from the subdivision. This gives us the following criterion, where $F(T,\varphi)$ denotes the face of $\Theta(n)$ corresponding to $(T,\varphi)$.

\begin{proposition}\label{contraction}
A face $F(T_1,\varphi_1)$ of $\Theta(n)$ is contained in another face $F(T_2,\varphi_2)$ if and only if $(T_1,\varphi_1)$ can be obtained from $(T_2,\varphi_2)$ by a series of contractions of non-leaf edges.
\end{proposition}

\section{Binary geometries}

We give an abstract definition for binary geometries, following~\cite{L}. For two vertices $i$ and $j$ in a simplicial complex we write $i\sim j$ to denote that $\{i,j\}$ is a face of the complex.

\begin{definition}
Consider a flag complex $\Delta$ with vertex set $V$ together with a \textit{compatibility degree} $a_{i,j}\in\bZ_{>0}$ for every pair $i,j\in V$ such that $i\not\sim j$. We say that the affine subvariety $\wt U\subset \bC^V$ cut out by the \textit{$u$-equations}
\begin{equation}\label{uequation}
R_i:=u_i+\prod_{j\not\sim i} u_j^{a_{i,j}}-1=0,\quad i\in V    
\end{equation}\label{binarydef}
is a \textit{binary geometry} (defined by $\Delta$ and the values $a_{i,j}$) if the following condition is satisfied. For a subset $S\subset V$ the intersection $\wt U_S=\wt U\cap \bC^{V\bs S}$ is nonempty if and only if $S$ is a face of $\Delta$ and, in this case, $\widetilde U_S$ is irreducible of dimension $\dim\Delta+1-|S|$. Here $\bC^{V\bs S}\subset\bC^V$ denotes the subspace of points satisfying $u_i=0$ for all $i\in S$.
\end{definition}

\begin{example}
    Let $\Delta$ be the 4-cycle with vertices $1$, $2$, $3$, $4$ and edges $\{1,2\}$, $\{2,3\}$, $\{3,4\}$, $\{4,1\}$. With all compatibility degrees 1, the resulting variety $\wt U$ is a binary geometry: the 2-plane in 4-space defined by $u_1 + u_3 = 1$ and $u_2 + u_4 = 1$.
\end{example}

Note that $\wt U_\varnothing=\wt U$, hence, $\wt U$ is irreducible of dimension $\dim\Delta+1$. Somewhat less trivially, one may show that if $F$ is a facet then $\wt U_F$ is a point. This follows from the fact that if $\wt U$ is a binary geometry, then $\Delta$ is pure (\cite[Corollary 2.7]{L}). Indeed, for every facet $F$ we have $|F|=\dim\Delta+1$ and, consequently, $\dim\wt U_F = 0$.

For $\wt U$ a binary geometry the symbol $U$ without a tilde will be used to denote the very affine part of $\wt U$: the intersection of $\wt U$ with the coordinate torus $(\bC^*)^V$. Similarly, $U_F=\wt U_F\cap (\bC^*)^{V \bs F}$ for a face $F$. One sees that $U_F$ is an open subset of $\wt U_F$. Furthermore, $\wt U_F$ is the disjoint union of the locally closed sets $U_G$ over all $G\supset F$. In particular, the collection of all $U_F$ forms a stratification of $\wt U$.

\section{$\cM_{0,n}$ and its tropicalization}
\label{M0nsec}

In this section, we recall the basic properties of $\cM_{0,n}$ and overview the known results concerning its tropicalization. 

As before, consider $n\ge 3$. We first define $\cM_{0,n}$ as the (very affine part of the) binary geometry defined by $\Delta(n)$ and then discuss more classical definitions. The vertices of $\Delta(n)$ can be enumerated by the set $D$ of pairs $(i,j)$ with $1\le i<j\le n$ and $i-j$ is not equal to $\pm1$ modulo $n$. Indeed, every vertex of $\Delta(n)$ corresponds to a diagonal of the $n$-gon $P$ with edges labeled in accordance with the ordering $(1,\dots,n)$. Let $i\in[n]$ denote the vertex of $P$ that is adjacent to the edges with labels $i$ and $i+1$ (the latter considered modulo $n$); denote by $d_{i,j}$ the diagonal connecting vertices $i$ and $j$. Evidently, $(i,j)\in D$. Hence, the binary geometry of $\Delta(n)$ can be viewed as a subvariety of the $n(n-3)/2$-dimensional affine space $\bC^D$. Applying Definition~\ref{binarydef} we obtain the following.

\begin{definition}
    \label{def:M0n-bg}
    $\wt\cM_{0,n}$ is the binary geometry defined by $\Delta(n)$ with all compatibility degrees equal to $1$. Explicitly, $\wt\cM_{0,n}$ is the affine subvariety of $\bC^D$ cut out by the $u$-equations
    $$R_{i,j} := u_{i,j} + \prod_{d_{k, l}\text{ crosses }d_{i,j}} u_{k,l} ~- 1=0,\quad (i,j)\in D.$$
    The ideal $I_n\subset\bC[u_{i,j}]_{(i,j)\in D}$ generated by the $R_{i,j}$ is the defining ideal of $\wt\cM_{0,n}$. The very affine variety $\cM_{0,n}$ is the open part $\wt\cM_{0,n}\cap(\bC^*)^D$.
\end{definition}

It is also true that the ideal $I_n$ is prime, i.e.\ it is the entire vanishing ideal of $\wt\cM_{0,n}$. Below we write $\Trop\cM_{0,n}$ to refer to the tropicalization of this ideal: the fan $\Trop I_n$ in the ambient real space $\bR^D$.

Classically, $\cM_{0,n}$ is defined as the moduli space of (ordered) $n$-tuples of distinct points in $\bP^1$ considered up to projective transformations. Arranging the $n$ points into a $2 \times n$ matrix lets one view $\cM_{0,n}$ as the quotient of $\Gr^\circ(2,n)$ (the Grassmannian of lines in $\PP^n$ with non-vanishing Pl\"ucker coordinates) by the torus action of $(\bC^\ast)^n$.

Here recall that for a point in the complex Grassmannian $\Gr(2,n)$ represented by a $2\times n$ matrix, its Pl\"ucker coordinates are the $2\times 2$ minors of the matrix. In particular, this realizes $\Gr(2,n)$ as a subvariety in $\bP(\bC^{n\choose 2})$. This subvariety is cut out by the \textit{Pl\"ucker ideal} $J_n\subset\bC[p_{i,j}]_{1\le i<j\le n}$. Thus, $\Gr^\circ(2,n)$ is the very affine part $\Gr(2,n)\cap(\bC^*)^{n\choose 2}$. We also have an embedding $(\bC^*)^n\hookrightarrow(\bC^*)^{n\choose 2}$ taking the point $(t_i)_{i\in[n]}$ to $(t_it_j)_{1\le i<j\le n}$. Since $(\bC^*)^{n\choose 2}$ naturally acts on $\bP(\bC^{n\choose 2})$, this embedding defines an action of $(\bC^*)^n$ on the same space. This $(\bC^*)^n$-action preserves $\Gr(2,n)$ and $\Gr^\circ(2,n)$, it is the torus action mentioned above. 

Furthermore, the defined $n$-dimensional subtorus of $(\bC^*)^{n\choose 2}$ is the largest subtorus preserving $\Gr(2,n)$. This means that the tropicalization of this subtorus is precisely the lineality space of $\Trop \Gr^\circ(2,n)=\Trop J_n$. We denote this $n$-dimensional lineality space by $L$. We deduce that $\Trop \cM_{0,n}$ is the quotient of $\Trop \Gr^\circ(2,n)$ modulo $L$ (as is well-known, see, for instance, \cite[Theorem 3.5]{GKM}). We denote the respective quotient map from $\bR^{n\choose 2}$ to $\bR^D$ by $q$. 

Definition \ref{def:M0n-bg} of $\cM_{0,n}$ is consistent with the definition of $\cM_{0,n}$ as a quotient of $\Gr^\circ(2,n)$, see \cite{Bro}. 
Given a point in $\cM_{0,n}$, i.e.\ an $n$-tuple $(x_1, \dots, x_n)$ of points in $\bP^1$, the coordinate $u_{i,j}$ is equal to the cross-ratio $(x_i,x_{i+1};x_j,x_{j+1})$. Therefore, the projection from $\Gr^\circ(2,n)$ to $\cM_{0,n}$ is given by the~formula
\begin{equation}
    u_{i,j} = \frac{p_{i,j+1} p_{i+1,j}}{p_{i,j} p_{i+1,j+1}}.
\end{equation}
Consequently, the quotient map $q$ from $\Trop \Gr^\circ(2,n)$ to $\Trop \cM_{0,n}$ is given by 
\begin{equation}
    \label{eqn:proj-trop}
    q(w)_{i,j} = w_{i,j+1} + w_{i+1,j} - w_{i,j} - w_{i+1,j+1},
\end{equation}
where the coordinates in the ambient real spaces of $\Trop \Gr^\circ(2,n)$ and $\Trop \cM_{0,n}$ are indexed by $\{(i,j)\}_{1 \le i < j \le n}$ and by $D$ respectively. 


A famous result of~\cite{SpSt} realizes $\Trop \Gr^\circ(2,n)$ as the product of the space of phylogenetic trees with an $n$-dimensional real space. Let $(T, \varphi)$ be a phylogenetic tree on $n$ leaves with set of non-leaf edges $E$. A \textit{realization} of $(T, \varphi)$ is a choice of lengths for the non-leaf edges given by a vector $l=(l_e)_{e\in E}$ in $\RR^E_{\ge0}$. Every realization induces a (pseudo)metric on $[n]$ by setting the distance $d_{(T,\varphi),l}(i,j)$ to be the sum of $l_e$ over the non-leaf edges in the unique path between $\varphi(i)$ and $\varphi(j)$ in $T$. The cone corresponding to $(T, \varphi)$, denoted $C^\prime_{T, \varphi}$, is a simplicial cone of dimension $|E|$ whose points encode the distance functions on $[n]$ that arise from realizations of $T$. More precisely,
\begin{equation*}
    C^\prime_{T, \varphi} = \{ (d_{(T,\varphi),l}(i,j))_{1\le i<j\le n} \mid l \in \RR^E_{\ge0} \} \subset \RR^{\binom{n}{2}}.
\end{equation*}

It is clear that the facets of $C^\prime_{T,\varphi}$ are precisely those cones $C^\prime_{T',\varphi'}$ for which $(T',\varphi')$ is obtained from $(T,\varphi)$ by contracting a non-leaf edge. This shows that together all of the cones $C^\prime_{T,\varphi}$ form a fan $\cT$. In view of Proposition~\ref{contraction}, we also have an isomorphism between the face posets of $\cT$ and $\Theta(n)$.\footnote{
For a polyhedral fan we do not consider the empty set to be a face.}
Combinatorially, $\cT$ is a fan over $\Theta(n)$, i.e.\ the space of phylogenetic trees.

Up to a sign (which is due to our choice of the min-convention) and lineality space, $\cT$ is the tropical Grassmannian. Here note that every cone $C^\prime_{T,\varphi}$ is transversal to $L$, hence we indeed obtain a fan of pure dimension $2n-3=\dim J_n$.

\begin{theorem}[{\cite[Theorems 3.4, 4.2]{SpSt}}]
    \label{thm:tropGr2n}
    The tropical Grassmannian $\Trop \Gr^\circ(2,n)$ consists of the cones $-C'_{T,\varphi}+L$ for $(T,\varphi)$ ranging over all phylogenetic trees with $n$ leaves.
\end{theorem}
 

As discussed above, $\Trop \cM_{0,n}$ is equal to $\Trop \Gr^\circ(2,n)$ modulo its lineality space $L$, hence it is also combinatorially equivalent to the space of phylogenetic trees. In other words, $\Trop \cM_{0,n}$ is formed by the cones $-q(C'_{T, \varphi})$. We apply~\eqref{eqn:proj-trop} and write these cones out explicitly using the shorthand $d=d_{(T,\varphi),l}$:
\begin{equation}\label{CTphiA}
C_{T,\varphi}=\{ (d(i,j)+d(i+1,j+1)-d(i+1,j)-d(i,j+1))_{(i,j)\in D} \mid l \in \RR^E_{\ge0} \}.    
\end{equation}

\begin{theorem}
    \label{thm:tropM0n}
    The tropicalization $\Trop \cM_{0,n}$ consists of the cones $C_{T,\varphi}$ for $(T,\varphi)$ ranging over all phylogenetic trees with $n$ leaves.
\end{theorem}



We now turn to the discussion of sign patterns and signed tropicalizations of $\cM_{0,n}$. This discussion is kept brief since it is logically independent from our results. Theorem~\ref{thm:tropM0n-pos} is only needed as a prototype that illustrates our proposed generalization to type C.

First of all, note that an element of the symmetric group $\fS_n$ acts naturally on $\cM_{0,n}$ by permuting the points of the respective $n$-tuple in $\bP^1$ (i.e.\ the columns of the $2\times n$ matrix). With respect to the coordinates $u_{i,j}$ every permutation acts by a monomial transformation, see~\cite[Subsection 11.2]{AHL}. This induces an $\fS_n$-action on $\Trop\cM_{0,n}$ by linear transformations of the ambient space $\bR^D$. Furthermore, $\fS_n$ also acts on the set of phylogenetic trees by $\sigma(T,\varphi)=(T,\varphi\circ\sigma^{-1})$ for $\sigma\in\fS_n$. By tracing the definitions one sees that this action is compatible with the action on $\Trop\cM_{0,n}$, i.e.\ $\sigma(C_{T,\varphi})=C_{\sigma(T,\varphi)}$.

There are $(n-1)!/2$ sign patterns occurring in $\cM_{0,n}(\bR)$, which are enumerated by dihedral orderings of $[n]$, again, see~\cite[Subsection 11.2]{AHL}.
For a dihedral ordering $\alpha$ we denote the corresponding sign pattern by $\tau_\alpha$\footnote{
The correspondence $\alpha\mapsto\tau_\alpha$ is not hard to describe. One considers $n$ distinct points $x_1,\dots,x_n$ in the real projective line ordered according to $\alpha$ and sets $(\tau_\alpha)_{i,j}$ equal to $\sgn(x_i,x_{i+1};x_j,x_{j+1})$.} 
and the set of points in $\cM_{0,n}(\bR)$ with coordinates fitting the sign pattern $\tau_\alpha$ by $\cM_{0,n}(\alpha)$. In fact, the subsets $\cM_{0,n}(\alpha)$ are precisely the connected components of $\cM_{0,n}(\bR)$. The natural action of $\fS_n$ on the set of dihedral orderings is compatible with its action on $\cM_{0,n}$: one has $\sigma(\cM_{0,n}(\alpha))=\cM_{0,n}(\sigma\alpha)$ so that the action permutes the connected components transitively. This means that the corresponding signed tropicalizations given by the $\tau_\alpha$ are also permuted by the $\fS_n$-action: $\Trop_{\tau_{\sigma\alpha}}\cM_{0,n}=\sigma(\Trop_{\tau_\alpha}\cM_{0,n})$. 

Finally, $\Trop_{>0}\cM_{0,n}$ is the fan over the dual-associahedron subcomplex $\Delta(n)$: it is formed by those cones $C_{T,\varphi}$ for which $(T,\varphi)$ is compatible with the dihedral ordering $(1,\dots,n)$, see~\cite[Section 3.12.1]{L}. Combining this with the above, we summarize as follows.

\begin{theorem}\label{thm:tropM0n-pos}
There is a bijection $\alpha\mapsto\tau_\alpha$ from the set of dihedral orderings of $[n]$ to the set of sign patterns occurring in $\cM_{0,n}(\bR)$ such that the following holds. For every dihedral ordering $\alpha$ the signed tropicalization $\Trop_{\tau_\alpha}\cM_{0,n}$ is the fan over the dual associahedron $\Delta(\alpha)$ formed by those cones $C_{T,\varphi}$ for which $(T,\varphi)$ is compatible with $\alpha$.
\end{theorem}

\section{The complex of axially symmetric phylogenetic trees}\label{ASsec}

Choose an integer $n\ge 3$ and consider the $2n$-element set $N=[1,n]\cup[-n,-1]$. As in Section~\ref{phyltrees} we may consider phylogenetic trees with leaves labeled by the set $N$ as well as dihedral orderings of the set $N$.

We say that a dihedral ordering of $N$ is \textit{axially symmetric} if it can be represented by an ordering $(\sigma_1,\dots,\sigma_{2n})$ of $N$ such that $\sigma_i=-\sigma_{2n+1-i}$ for any $i\in[2n]$. There are $2^{n-2}n!$ such dihedral orderings. Let $\alpha$ be an axially symmetric dihedral ordering (or ASDO) and consider a regular $2n$-gon $P$ with its edges labeled by $N$ in accordance with $\alpha$. This labeling distinguishes an axis of symmetry $l$ among the longest diagonals of $P$ reflection across which takes the edge labeled by $i$ to the edge labeled by $-i$. 

Furthermore, for every diagonal $d$ of $P$ we have a diagonal $d'$ that is symmetric to $d$ with respect to $l$. The two diagonals $d$ and $d'$ coincide if and only if $d$ is $l$ or is one of the $n-1$ diagonals perpendicular to $l$. This allows us to consider \textit{axially symmetric subdivisions} of $P$: subdivisions which contain a diagonal if and only if they also contain the one symmetric to it. Such a subdivision may not contain diagonals that cross $l$ but are not perpendicular to $l$. 

Similarly to Section~\ref{phyltrees}, a subdivision of $P$ defines a phylogenetic tree with leaves labeled by the set $N$. We call a phylogenetic tree \textit{axially symmetric} if it arises in this way from an ASDO $\alpha$ and an axially symmetric subdivision of $P$. In this case we also say that the resulting tree is \textit{compatible} with $\alpha$. Note that an  axially symmetric phylogenetic tree (ASPT) will be compatible with multiple ASDOs.

There are $(n+2)(n-1)/2$ coarsest nontrivial axially symmetric subdivisions of $P$: the one formed by $l$, those formed by each of the $n-1$ diagonals perpendicular to $l$ and $(n+1)(n-2)/2$ more formed by a pair of diagonals that are symmetric to each other and do not cross $l$. We say that an ASPT compatible with $\alpha$ is \textit{minimal} if it arises from a coarsest nontrivial axially symmetric subdivision of $P$. Such a tree will have two or three non-leaf vertices.

\begin{definition}
We define a simplicial complex $\Delta_\AS(\alpha)$ whose vertex set is the set of minimal ASPTs compatible with $\alpha$. A subset of vertices forms a face if and only if the corresponding axially symmetric subdivisions have a common refinement.
\end{definition}

We see that the faces of $\Delta_\AS(\alpha)$ are enumerated by the axially symmetric subdivisions of $P$ or, alternatively, by the ASPTs that are compatible with $\alpha$.

\begin{proposition}
The simplicial complex $\Delta_\AS(\alpha)$ is isomorphic to the dual associahedron $\Delta(n+2)$.
\end{proposition}
\begin{proof}
We establish a bijection respecting the order of refinement between the axially symmetric subdivisions of $P$ and all subdivisions of an $(n+2)$-gon. Let $Q'$ be one of the $(n+1)$-gons into which $P$ is divided by $l$. To $Q'$ we add a vertex $v$ adjacent to the endpoints of $l$ to obtain a convex $(n+2)$-gon $Q$. Now choose an axially symmetric subdivision $S$ of $P$ and consider the subdivision $R$ of $Q$ defined as follows. $R$ includes all of the diagonals in $S$ whose endpoints lie in $Q'$ (each such diagonal is either $l$ or a diagonal of $Q'$). In addition, for every diagonal $d$ in $S$ that is perpendicular to $l$ the subdivision $R$ includes the diagonal connecting $v$ with the endpoint of $d$ that lies in $Q'$. One easily checks that this is a bijection with the desired property. An example is shown below.
\end{proof}

\scalebox{.85}{
\begin{minipage}{.37\textwidth}
    \centering
    \include{figures/octogonToHexagon1}
\end{minipage}%

\begin{minipage}{.37\textwidth}
    \centering
    \include{figures/octogonToHexagon2}
\end{minipage}%

\begin{minipage}{.37\textwidth}
    \centering
    \include{figures/octogonToHexagon3}
\end{minipage}
}
\vspace{-4mm}
    

Now we can give an analogue of Definition~\ref{coptdef}. Note that, similarly to the phylogenetic trees considered in Section~\ref{phyltrees}, any ASPT is compatible with multiple ASDOs. Hence, the union of the various $\Delta_\AS(\alpha)$ is not disjoint.
\begin{definition}
The \textit{complex of axially symmetric phylogenetic trees} $\Theta_\AS(n)$ is the union of $\Delta_\AS(\alpha)$ over all ASDOs $\alpha$, with vertices represented by the same phylogenetic trees identified. The vertex set of the simplicial complex $\Theta_\AS(n)$ is the set of all minimal ASPTs. A subset of vertices forms a face if and only if it forms a face in one of the $\Delta_\AS(\alpha)$.    
\end{definition}

We see that the faces of $\Theta_\AS(n)$ are enumerated by all ASPTs with leaves labeled by $N$. Since every $\Delta_\AS(\alpha)$ is a flag complex of pure dimension $n-2$, so is $\Theta_\AS(n)$. 

\begin{example}\label{ASPTexample}
For $n = 3$ the complex $\Theta_{\AS}(n)$ is shown in Figure~\ref{C2graph}. It is one-dimensional with 13 vertices and 21 edges. The minimal ASPTs for each vertex are given in Table~\ref{C2trees}. Three of the edges are labeled with the corresponding trees in Figure~\ref{C2graph}. The subcomplex $\Delta_{\AS}(\alpha)$ for $\alpha$ given by $(1,2,3,-3,-2,-1)$ is highlighted in red; it is isomorphic to $\Delta(5)$, i.e.\ a 5-cycle.    
\end{example}

For an ASPT $(T,\varphi)$ (here $\varphi$ is a bijection from $N$ to the leaf set of $T$) let $F(T,\varphi)$ denote the face of $\Theta_\AS(n)$ corresponding to $(T,\varphi)$. A containment of faces $F(T_1,\varphi_1)\subset F(T_2,\varphi_2)$ can be interpreted similarly to Proposition~\ref{contraction}. 

Let $\alpha$, $P$ and $l$ be as above and let $(T,\varphi)$ be compatible with $\alpha$. Since $(T,\varphi)$ arises from an axially symmetric subdivision of $P$, one has an involution $\iota$ on $T$ that corresponds to reflecting the subdivision across $l$. We call $\iota$ the \textit{symmetry} of $(T,\varphi)$ and for an edge $e$ of $T$ we say that $\iota(e)$ is \textit{symmetric} to $e$ (note that we may have $e=\iota(e)$). Alternatively, $\iota$ can be characterized as the unique automorphism of $T$ that exchanges the vertices $\varphi(i)$ and $\varphi(-i)$ for all $i\in N$. This means that the symmetry does not depend on the choice of $\alpha$: if $(T,\varphi)$ is also compatible with $\beta$, then the respective subdivision will provide the same involution.

For a non-leaf edge $e$ of $T$ we may again consider the phylogenetic tree obtained from $(T,\varphi)$ by contracting $e$, but the result might not be axially symmetric. However, if we contract both $e$ and $\iota(e)$ (or just $e$ if it is symmetric to itself), then what we obtain is axially symmetric. We call such an operation a \textit{symmetric edge contraction}. Since it corresponds to the deletion of a pair of mutually symmetric diagonals from the subdivision (or of a single diagonal symmetric to itself), we have the following.
\begin{proposition}\label{contractionC}
A face $F(T_1,\varphi_1)$ of $\Theta_\AS(n)$ is contained in another face $F(T_2,\varphi_2)$ if and only if $(T_1,\varphi_1)$ can be obtained from $(T_2,\varphi_2)$ by a series of symmetric edge contractions.
\end{proposition}

In particular, one may check that the above holds for $F(T_2,\varphi_2)$ one of the labeled edges in Figure~\ref{C2graph} and $F(T_1,\varphi_1)$ one of its endpoints.

Next, for $(T,\varphi)$ and $\iota$ as above let $E$ denote the set of non-leaf edges in $T$ and let $k(T,\varphi)$ be the number of $\iota$-orbits in $E$. In other words, 
\[k(T,\varphi)=|E|-|\{e\in E| \iota(e)\neq e\}|/2.\] 
The previous proposition has the following consequence.
\begin{corollary}\label{facedimC}
The dimension of the face $F(T,\varphi)$ is $k(T,\varphi)-1$.        
\end{corollary}

\section{The (sub)complex of centrally symmetric phylogenetic trees}\label{CSsec}

A dihedral ordering $\alpha$ of $N$ is \textit{centrally symmetric} if it is given by an ordering $(\sigma_1,\dots,\sigma_{2n})$ of $N$ satisfying $\sigma_{i+n}=-\sigma_i$ for all $i\in[n]$ (hence, any other ordering representing $\alpha$ also has this property). There are $2^{n-2}(n-1)!$ centrally symmetric dihedral orderings (or CSDOs).

Let $P$ be a regular $2n$-gon with edges labeled according to a CSDO $\alpha$. A subdivision of $P$ is \textit{centrally symmetric} if with every diagonal $d$ it also includes the diagonal centrally symmetric to $d$. By considering phylogenetic trees arising from such subdivisions, we obtain a notion of \textit{centrally symmetric phylogenetic trees} (CSPTs) and a notion of compatibility between CSPTs and CSDOs.

There are $n(n-1)$ coarsest nontrivial centrally symmetric subdivisions of $P$: the $n$ formed by one longest diagonal and $n(n-2)$ more formed by a pair of distinct diagonals that are centrally symmetric to each other. Again, we say that a CSPT compatible with $\alpha$ is \textit{minimal} if it arises from a coarsest nontrivial centrally symmetric subdivision of $P$. Such trees also have two or three non-leaf vertices.

\begin{definition}
We define a simplicial complex $\Delta_\CS(\alpha)$ whose vertex set is the set of minimal CSPTs compatible with $\alpha$. A subset of vertices forms a face if and only if the corresponding centrally symmetric subdivisions have a common refinement.
\end{definition}

The faces of $\Delta_\CS(\alpha)$ are in bijection with centrally symmetric subdivisions of $P$ or, in other words, with CSPTs compatible with $\alpha$. This is a flag complex of pure dimension $n-2$; we refer to $\Delta_\CS(\alpha)$ as a \textit{dual cyclohedron} since its face poset is dual to that of a cyclohedron. Again, a CSPT will be compatible with multiple CSDOs, making the union of the $\Delta_\CS(\alpha)$ non-disjoint.

\begin{definition}
The \textit{complex of centrally symmetric phylogenetic trees} $\Theta_\CS(n)$ is the union of $\Delta_\CS(\alpha)$ over all CSDOs $\alpha$, with vertices represented by the same phylogenetic trees identified. This is the simplicial complex whose vertex set consists of all minimal CSPTs with a subset forming a face if and only if it forms a face in one of the $\Delta_\CS(\alpha)$.    
\end{definition}

Similarly to the previous cases, the faces of $\Theta_\CS(n)$ are in bijection with CSPTs. An important observation is that CSPTs are also axially symmetric.

\begin{proposition}
Every CSPT is an ASPT.
\end{proposition}
\begin{proof}
Let $(T,\varphi)$ be a CSPT compatible with a CSDO $\alpha$ and $P$ be a regular $2n$-gon with edges labeled in accordance with $\alpha$. Let $S$ be the centrally symmetric subdivision of $P$ corresponding to $(T,\varphi)$. We claim that there exists a longest diagonal $l$ of $P$ that does not cross any of the diagonals in $S$. Indeed, if $S$ includes a longest diagonal $l'$ then we just set $l=l'$. Otherwise, the center of $P$ is not contained in any of the diagonals in $S$, hence it is an interior point of some two-dimensional cell $Q$ in $S$. Note that $Q$ is itself centrally symmetric, hence any longest diagonal ending in a vertex of $Q$ can be chosen as $l$.

Now we transform $S$ to obtain an axially symmetric subdivision $R$. The idea is to choose one of the two $(n+1)$-gon halves formed by $l$ and ``flip'' both the edge labels and the diagonals in this half (see the figure below). More precisely, choose an ordering $(\sigma_1,\dots,\sigma_{2n})$ of $N$ representing $\alpha$ so that the edges labeled by $\sigma_1,\dots,\sigma_n$ lie on one side of $l$ and the remaining edges lie on the other. The ordering $(\sigma_1,\dots,\sigma_n,\sigma_{2n},\dots,\sigma_{n+1})$ (here the last $n$ elements are reversed) defines an ASDO $\beta$. If in $P$ we reverse the order of the edge labels lying on one side of $l$, the edges will be labeled in accordance with $\beta$. Now we replace each diagonal in $S$ lying on the same side of $l$ by its reflection in the perpendicular bisector of $l$ to obtain the subdivision $R$. It is clear that $R$ is an axially symmetric subdivision and the corresponding ASPT compatible with $\beta$ equals $(T,\varphi)$.
\end{proof}

\begin{center}
\scalebox{0.9}
{
    \centering
    \begin{minipage}{.35\textwidth}
        \centering
        \include{figures/centralToAxial1}
    \end{minipage}%
    \begin{minipage}{.35\textwidth}
        \centering
        \include{figures/centralToAxial2}
    \end{minipage}%
    \begin{minipage}{.35\textwidth}
        \centering
        \include{figures/centralToAxial3}
    \end{minipage}
}
\end{center}
\vspace{-7mm}

In the proof we define a procedure which constructs an axially symmetric subdivision from a centrally symmetric subdivision. This procedure depends on the choice of a longest diagonal $l$. However, it is clear that if one centrally symmetric subdivision refines another, then there is a choice of $l$ compatible with both of them. Under this choice of $l$ the respective axially symmetric subdivisions will satisfy the same refinement relation. It follows that a minimal CSPT is also minimal as an ASPT, i.e.\ it is a vertex of $\Theta_\AS(n)$. More generally, if a set of centrally symmetric subdivisions of $P$ have a common refinement, then there is a choice of $l$ compatible with all of them and the respective axially symmetric subdivisions also have a common refinement. We obtain the following.

\begin{corollary}
For every CSDO $\alpha$ the complex $\Delta_\CS(\alpha)$ is a subcomplex of $\Theta_\AS(n)$, the complex of ASPTs. Consequently, the complex of CSPTs $\Theta_\CS(n)$ is also subcomplex of $\Theta_\AS(n)$.
\end{corollary}

In Figure~\ref{C2graph} the subcomplex $\Theta_\CS(3)$ is highlighted in bold. It has 10 vertices and 12 edges. One may check that the respective trees are indeed centrally symmetric. The subcomplex $\Delta_\CS(\alpha)$ for $\alpha$ given by $(1,-2,3,-1,2,-3)$ is highlighted in blue; it is the dual cyclohedron, i.e.\ a 6-cycle. 

\section{$\cM_{\rC_{n-1}}$ and its tropicalization}

In this section we first recall the definition of the type C cluster configuration space $\cM_{\rC_{n-1}}$. We follow \cite{AHL} where $\cM_{\rC_{n-1}}$ is realized as the (very affine part of the) binary geometry defined by the dual cyclohedron. We then give a conjectural description of $\Trop \cM_{\rC_{n-1}}$ as of the space of ASPTs. We also conjecture that the signed tropicalizations of $\cM_{\rC_{n-1}}$ are the subfans corresponding to the dual-associahedron and the dual-cyclohedron subcomplexes.

Let $n$ and $N$ be as in the previous section and let $\alpha_0$ denote the CSDO of $N$ given by the ordering $(1,\dots,n,-1,\dots,-n)$. We set $\Delta_\CS(n)=\Delta_\CS(\alpha_0)$. As in the case of $\cM_{0,n}$, we enumerate the vertices of $\Delta_\CS(n)$ by certain pairs of elements of $N$. Recall that these vertices correspond to minimal CSPTs compatible with $\alpha_0$ and, thus, to coarsest nontrivial centrally symmetric subdivisions of the regular $2n$-gon $P$ with edges labeled in accordance with $\alpha_0$. 

For $i\in N$ we use the notation $i\pp$ for the element that succeeds $i$ in $\alpha_0$. Explicitly: $i\pp=\sgn(i)(|i|+1)$ for $|i|\in[1,n-1]$ while $(\pm n)\pp=\mp1$.
Let $i\in N$ denote the vertex of $P$ that is adjacent to edges with labels $i$ and $i\pp$. Denote the diagonal connecting vertices $i$ and $j$ by $d_{i,j}$. A coarsest nontrivial centrally symmetric subdivision is either formed by a single diagonal of the form $d_{i,-i}$ or by a pair of non-crossing diagonals of the forms $d_{i,j}$ and $d_{-i,-j}$. In the former case the subdivision corresponds to the pair $(i,-i)$ where $i\in[n]$. In the latter case it corresponds to the pair $(i,j)$ where we may assume that $i\in[n]$ and $i<|j|$. The resulting set of pairs $D$ enumerating the vertices of $\Theta_\CS(n)$ consists of pairs $(i,j)$ of distinct elements of $N$ for which $i\in[n]$, $i\le |j|$ and, furthermore, both $i\pp\neq j$ and $i\neq j\pp$. In particular, $|D|=n(n-1)$. We now apply Definition~\ref{binarydef} with the compatibility degree $a_{(i,j),(k,l)}$ of two non-adjacent vertices in $\Theta_\CS(n)$ equal to the number of diagonals in the set $\{d_{k,l},d_{-k,-l}\}$ that cross the diagonal $d_{i,j}$ (thus either 1 or 2).

\begin{definition}[{\cite[Section 3.2.2]{AHL}}]
\label{MCndef}
$\wt\cM_{\rC_{n-1}}$ is the affine subvariety of $\bC^D$ cut out by the $u$-equations
$$R_{i,j} := u_{i,j} + \prod_{\substack{d_{k, l}\text{ or }d_{-k,-l}\\\text{ crosses }d_{i,j}}} u_{k,l}^{a_{(i,j),(k,l)}} ~- 1=0,\quad (i,j)\in D.$$
The ideal $I_n\subset\bC[u_{i,j}]_{(i,j)\in D}$ generated by the $R_{i,j}$ is the defining ideal of $\wt\cM_{\rC_{n-1}}$. The very affine variety $\cM_{\rC_{n-1}}$ is the open part $\wt\cM_{\rC_{n-1}}\cap(\bC^*)^D$.
\end{definition}

\begin{example}\label{C2example}
Let $n=3$. In this case we have
\[D=\{(1,-1),(2,-2),(3,-3),(1,3),(1,-2),(2,-3)\}.\]
The $u$-equations defining $\wt\cM_{\rC_2}$ are
\[u_{1,-1}+u_{2,-2}u_{3,-3}u_{2,-3}^2-1=0,\quad u_{1,3}+u_{2,-2}u_{1,-2}u_{2,-3}-1=0\]
together with with 4 more equations obtained from the above by cyclic rotations, i.e.\ iterations of the substitution $u_{i,j}\mapsto u_{i\pp,j\pp}$ (with the subscripts on the right transposed or negated if necessary).
\end{example}

Our conjectured description of $\Trop\cM_{\rC_{n-1}}=\Trop I_n$ will be given in terms of pseudometrics on symmetric phylogenetic trees. Consider an ASPT $(T,\varphi)$, and let $E$ denote the set of non-leaf edges of $T$ and $\iota$ denote the symmetry of $(T,\varphi)$. A \textit{symmetric realization} of $(T,\varphi)$ is an assignment of lengths to the edges in $E$ given by a vector $l=(l_e)_{e\in E}\in\bR_{\ge 0}^E$ such that $l_e=l_{\iota(e)}$ for any $e\in E$. Such a realization determines distances between pairs of vertices of $T$ where the lengths of leaf edges are assumed to be zero. For $i,j\in N$ we denote by $d_{(T,\varphi),l}(i,j)$ the resulting distance function between leaf vertices $\varphi(i)$ and $\varphi(j)$. Evidently, this function satisfies $d(i,j)=d(-i,-j)$. 

Let $\SR(T,\varphi)$ denote the set of all symmetric realizations of $(T,\varphi)$. We associate a cone in $\bR^D$ with the tree $(T,\varphi)$ as follows:
\begin{equation}\label{CTphiC}
C_{T,\varphi}=
\{(d(i,j)+d(i\pp,j\pp)-d(i,j\pp)-d(i\pp,j))_{(i,j)\in D}|
l\in\SR(T,\varphi)\}       
\end{equation}
where we denote $d=d_{(T,\varphi),l}$.

\begin{proposition}
The set $C_{T,\varphi}$ is a simplicial cone of dimension $k(T,\varphi)$.
\end{proposition}
\begin{proof}
The set $\SR(T,\varphi)$ is naturally a cone in the space $\bR^E$. It is simplicial of dimension $k(T,\varphi)$ since a symmetric realization is determined by $k(T,\varphi)$ edge lengths. Formula~\eqref{CTphiC} provides a linear map $f:\bR^E\to\bR^D$ such that $f(\SR(T,\varphi))=C_{T,\varphi}$ and it suffices to check that $f$ is injective.

This can be deduced from the discussion in Section~\ref{M0nsec}. Let $\delta:[2n]\to N$ be the bijection taking $i\le n$ to $i$ and $i>n$ to $n-i$. We obtain a phylogenetic tree $(T,\varphi\circ\delta)$ with leaves labeled by $[2n]$. Realizations of $(T,\varphi\circ\delta)$ form an $|E|$-dimensional cone $K$ in $\bR^E$ (the positive orthant). The corresponding cone $C_{T,\varphi\circ\delta}$ defined by~\eqref{CTphiA} is an $|E|$-dimensional cone in $R^{D_{2n}}$ where $D_{2n}$ is the set of pairs $(i,j)$ with $1\le i<j\le 2n$ and $i-j$ not equal to $\pm1$ modulo $2n$. The cone $C_{T,\varphi\circ\delta}$ is the image of $K$ under the injective map $f_{2n}:\bR^E\to\bR^{D_{2n}}$ given by formula~\eqref{CTphiA}.

Now, by comparing formulas~\eqref{CTphiA} and~\eqref{CTphiC} it is not hard to define a map $g:\bR^D\to\bR^{D_{2n}}$ satisfying $f_{2n}=g\circ f$ which will imply the injectivity of $f$. Indeed, the elements of $D_{2n}$ correspond to all diagonals of the regular $2n$-gon while elements of $D$ correspond to a subset of these diagonals containing exactly one element of every central symmetry orbit. Thus, we have a natural surjection $\pi:D_{2n}\to D$ and set $g(w)_{i,j}=w_{\pi(i,j)}$ for $w\in\bR^D$.
\end{proof}

\begin{corollary}
The facets of $C_{T,\varphi}$ are the cones $C_{T',\varphi'}$ such that $(T',\varphi')$ is obtained from $(T,\varphi)$ by a symmetric edge contraction.
\end{corollary}
\begin{proof}
The facets of the cone $\SR(T,\varphi)$ are the subsets $F_e$, $e\in E$ of symmetric realizations $l$ with $l_e=0$. Note that $F_e=F_{\iota(e)}$. Choose $e\in E$ and let $(T',\varphi')$ be obtained from $(T,\varphi)$ by contracting $e$ and $\iota(e)$ (or just $e$ if $e=\iota(e)$). Then, under the linear bijection from $\SR(T,\varphi)$ to $C_{T,\varphi}$, the facet $F_e$ is mapped to $C_{T',\varphi'}$. This follows from the fact that we have a bijection $l\mapsto l'$ from $F_e$ to $\SR(T'\varphi')$ such that $d_{(T,\varphi),l}=d_{(T',\varphi'),l'}$. Here $l'$ is obtained from $l$ by simply forgetting the coordinates corresponding to $e$ and $\iota(e)$.
\end{proof}

The above shows that the cones $C_{T,\varphi}$ with $(T,\varphi)$ ranging over all ASPTs form a simplicial polyhedral fan in $\bR^D$. Furthermore, in view of Proposition~\ref{contractionC} and Corollary~\ref{facedimC}, this fan is combinatorially equivalent to the fan over the complex of ASPTs $\Theta_\AS(n)$. We refer to this fan as the \textit{space of axially symmetric phylogenetic trees}. We now state our first conjecture.
\begin{conjecture}\label{mainconj}
The tropicalization $\Trop\cM_{\rC_{n-1}}$ is the space of ASPTs: the fan formed by the cones $C_{T,\varphi}$ with $(T,\varphi)$ ranging over all ASPTs.
\end{conjecture}

Our second conjecture concerns the sign patterns occurring in $\cM_{\rC_{n-1}}(\bR)$ and the corresponding signed tropicalizations. In the following note that a dihedral ordering of $N$ cannot be both axially symmetric and centrally symmetric, thus there are $2^{n-2}(n+1)(n-1)!$ dihedral orderings with one of these properties. 
\begin{conjecture}\label{signedconj}
There is a bijection $\alpha\mapsto\tau_\alpha$ from the set of all dihedral orderings of $N$ that are either axially symmetric or centrally symmetric to the set of sign patterns occurring in $\cM_{\rC_{n-1}}(\bR)$ such that the following holds. 
\begin{enumerate}[label=(\alph*)]
\item If $\alpha$ is axially symmetric, then the signed tropicalization $\Trop_{\tau_\alpha} \cM_{\rC_{n-1}}$ is a fan over the dual associahedron $\Delta_\AS(\alpha)$: it is formed by those cones $C_{T,\varphi}$ for which $(T,\varphi)$ is an ASPT compatible with $\alpha$. 
\item If $\alpha$ is centrally symmetric, then the signed tropicalization $\Trop_{\tau_\alpha} \cM_{\rC_{n-1}}$ is a fan over the dual cyclohedron $\Delta_\CS(\alpha)$: it is formed by those cones $C_{T,\varphi}$ for which $(T,\varphi)$ is a CSPT compatible with $\alpha$.
\end{enumerate}
\end{conjecture}

The conjecture above extends a conjecture made in~\cite[Subsection 11.3]{AHL} where the authors, based on computational data, hypothesize that the number of sign patterns occurring in $\cM_{\rC_{n-1}}(\bR)$ is $2^{n-2}(n+1)(n-1)!$.

The case of the positive tropicalization, i.e.\ of the sign pattern $(1,\dots,1)$, is also studied in loc.\ cit. We expect the proposed bijection to give $\tau_{\alpha_0}=(1,\dots,1)$. This would agree with \cite[Theorem 9.2]{AHL} which implies that $\Trop_{>0} \cM_{\rC_{n-1}}$ is combinatorially equivalent to a fan over the dual-cyclohedron complex $\Delta_\CS(n)$.

\section{Example: $\Trop\cM_{\rC_2}$}\label{C2section}

The description of $\Trop \cM_{\rC_{2}}$ proposed by Conjecture~\ref{mainconj} is illustrated in this section by Figure~\ref{C2graph} and Table~\ref{C2trees}, below we explain how this data is to be read. The resulting fan does indeed coincide with $\Trop \cM_{\rC_{2}}$ as verified by the authors using the \texttt{Tropical} package in Macaulay2 (\cite{M2}).

The vertices in Figure~\ref{C2graph} represent rays of the fan while edges represent two-dimensional cones. For every minimal ASPT $(T,\varphi)$ in Table~\ref{C2trees}, below it we write the generator of the corresponding ray $C_{T,\varphi}$ as determined by \eqref{CTphiC}. Here $e_{i,j}$ denotes the basis vector corresponding to the coordinate $u_{i,j}$.

Furthermore, to illustrate Conjecture~\ref{signedconj}, one may check that the subfan given by the blue 6-cycle is precisely $\Trop_\tau \cM_{\rC_2}$ for $\tau=(1,1,1,1,-1,1)$ where we use the order of coordinates from Example~\ref{C2example}. Computationally, to see if a cone $C_{T,\varphi}$ lies in $\Trop_\tau \cM_{\rC_2}$ one may compute the initial ideal $\init_w\varepsilon_\tau(I_2)$ for $w$ an interior point of $C_{T,\varphi}$, e.g.\ by using the \texttt{leadTerm} function in Macaulay2. 
Similarly, one may also check that the subfan given by the red 5-cycle is precisely $\Trop_\tau \cM_{\rC_2}$ for $\tau=(1,1,-1,1,1,1)$.

{\begin{minipage}{1\textwidth}
    \vspace{-12mm}
    \include{figures/treePics}
    \centering
    \include{figures/C2graph}
    \vspace{-13mm}
    \renewcommand{\thefigure}{\Alph{figure}}
    
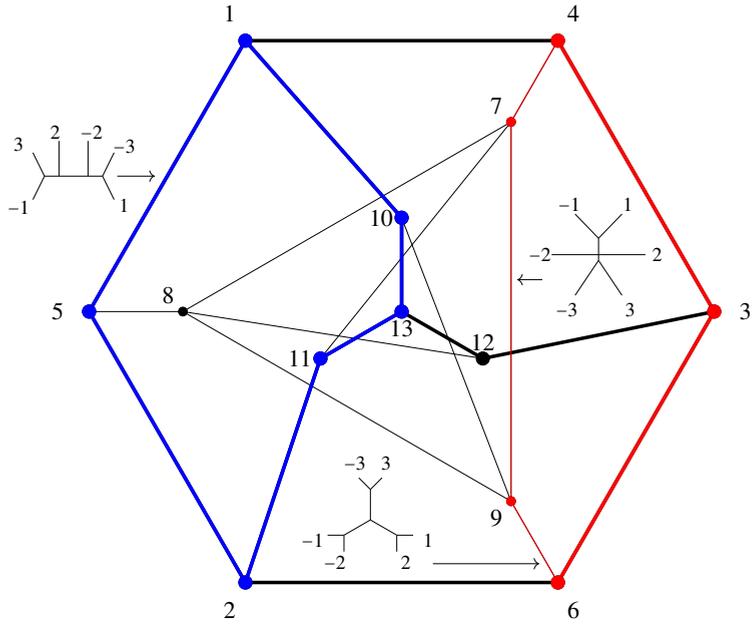
\captionof{figure}{the complex $\Theta_\AS(3)$.}\label{C2graph}
\end{minipage}}

{\begin{minipage}{1\textwidth}
    \include{figures/treePics}
    \vspace{-3mm}
    \centering

    \scalebox{.875}{\begin{tabular}{cccc}
         \begin{tikzpicture} \pic at (0,0) [scale=.75] {atree=-3/1/-2/2/-1/3/1}; \end{tikzpicture} & 
         \begin{tikzpicture} \pic at (0,0) [scale=.75] {ctree=1/2/3/-1/-3/-2/4}; \end{tikzpicture} & 
         \begin{tikzpicture} \pic at (0,0) [scale=.75] {btree=-2/-1/1/2/3/-3/7}; \end{tikzpicture} &
         \begin{tikzpicture} \pic at (0,0) [scale=.75] {ctree=-3/1/-2/3/2/-1/10}; \end{tikzpicture} \\[-2mm]
         
         $e_{1,-1}$ & $e_{1,3}$ & $e_{1,3} - e_{1,-1} - e_{3,-3}$ & $2e_{1,-1} - e_{1,3} - e_{1,-2}$ \\[2mm]
         
         \begin{tikzpicture} \pic at (0,0) [scale=.75] {atree=1/-3/2/-2/3/-1/2}; \end{tikzpicture} & 
         \begin{tikzpicture} \pic at (0,0) [scale=.75] {ctree=2/-3/1/-2/-1/3/5}; \end{tikzpicture} & 
         \begin{tikzpicture} \pic at (0,0) [scale=.75] {btree=-1/-2/2/1/3/-3/8}; \end{tikzpicture} & 
         \begin{tikzpicture} \pic at (0,0) [scale=.75] {ctree=1/2/-3/-1/3/-2/11}; \end{tikzpicture} \\[-2mm]
         
         $e_{2,-2}$ & $e_{1,-2}$ & $e_{1,-2} - e_{1,-1} - e_{2,-2}$ & $2e_{2,-2} - e_{1,-2} - e_{2,-3}$ \\[2mm]
         
         \begin{tikzpicture} \pic at (0,0) [scale=.75] {atree=1/2/3/-3/-2/-1/3}; \end{tikzpicture} &
         \begin{tikzpicture} \pic at (0,0) [scale=.75] {ctree=-3/1/2/3/-2/-1/6}; \end{tikzpicture} & 
         \begin{tikzpicture} \pic at (0,0) [scale=.75] {btree=-2/-3/3/2/1/-1/9}; \end{tikzpicture} &
         \begin{tikzpicture} \pic at (0,0) [scale=.75] {ctree=2/1/3/-2/-3/-1/12}; \end{tikzpicture} \\[-2mm]
         
         $e_{3,-3}$ & $e_{2,-3}$ & $e_{2,-3} - e_{2,-2} - e_{3,-3}$ & $2e_{3,-3} - e_{1,3} - e_{2,-3}$
    \end{tabular}}

    \hspace{-2.25cm}
    \scalebox{.875}{\begin{tabular}{cc}
        \begin{tikzpicture} \pic at (0,0) [scale=.75] {atree=1/-2/3/-3/2/-1/13}; \node at (4.75,0){$e_{1,-1} + e_{2,-2} + e_{3,-3} - e_{1,3} - e_{1,-2} - e_{2,-3}$}; \end{tikzpicture} &
    \end{tabular}}
    \setcounter{table}{1}
    \renewcommand{\thetable}{\Alph{table}}
    \captionof{table}{ASPTs and ray generators corresponding to vertices of $\Theta_\AS(3)$.}
    \label{C2trees}
\end{minipage}}

\bigskip \bigskip

\noindent {\bf Acknowledgments.} We thank Thomas Lam for introducing us to these topics and for his insightful suggestions. Thanks also to Karin Baur, Zachary Greenberg and Giulio Salvatori for helpful discussions, and to Chiara Meroni and Hannah Tillman-Morris for their feedback on the paper.

\begin{small}

\end{small}

\end{document}

%% file: figures/octogonToHexagon1.tex
\begin{tikzpicture}[scale=1/2]
    \foreach \i in {0, 1, 2, 3, 4, 5, 6, 7} {
        \coordinate (a\i) at (45*\i:3);
        \filldraw (a\i) circle (2pt);
        \coordinate (l\i) at ($.55*(45*\i:3) + .55*(45*\i+45:3)$);
    }

    \draw (a0)--(a1)--(a2)--(a3)--(a4)--(a5)--(a6)--(a7)--cycle;

    \draw[dashed] ($1.25*(a1)$)--($1.25*(a5)$);
    \node[right] at ($1.15*(a1)$){$l$};

    \draw (a4)--(a6);
    \draw (a6)--(a1)--(a4);
    \draw (a4)--(a2);
    \draw (a6)--(a0);

    \draw[->] ($1.5*(a0)$)--($2*(a0)$);
\end{tikzpicture}

%% file: figures/octogonToHexagon2.tex
\begin{tikzpicture}[scale=1/2]
    \foreach \i in {0, 1, 5, 6, 7} {
        \coordinate (a\i) at (45*\i:3);
        \filldraw (a\i) circle (2pt);
        \coordinate (l\i) at ($.55*(45*\i:3) + .55*(45*\i+45:3)$);
    }

    \draw (a5)--(a6)--(a7)--(a0)--(a1)--cycle;

    \draw[dashed] ($1.25*(a1)$)--($1.25*(a5)$);
    \node[right] at ($1.15*(a1)$){$l$};

    \draw (a6)--($.5*(a6) + .5*(a4)$);
    \draw (a6)--(a1);
    \draw (a6)--(a0);


    \draw[thick] (a1)--(a5)--(a6)--(a7)--(a0)--cycle;
    \node[right] at (l7){$Q'$};

    \draw[->] ($1.5*(a0)$)--($2.0*(a0)$);
\end{tikzpicture}

%% file: figures/octogonToHexagon3.tex
\begin{tikzpicture}[scale=1/2]

    \foreach \i in {0, 1, 5, 6, 7} {
        \coordinate (a\i) at (45*\i:3);
        \filldraw (a\i) circle (2pt);
    }

    \draw (a5)--(a6)--(a7)--(a0)--(a1);
    
    \coordinate (v) at ($.5*(a3)$);
    \filldraw (v) circle (2pt);
    \node[above left] at (v){$v$};

    \draw (v)--(a6);
    \draw (v)--(a1);
    \draw (v)--(a5);
    \draw (a0)--(a6);
    \draw (a1)--(a6);

     \draw[dashed] ($1.25*(a1)$)--($1.25*(a5)$);

    \node[right] at (l7){$Q$};
\end{tikzpicture}

%% file: figures/centralToAxial1.tex
\begin{tikzpicture}[scale=1/2]
    \foreach \i in {0, 1, 2, 3, 4, 5, 6, 7} {
        \coordinate (a\i) at (45*\i:3);
        \coordinate (l\i) at ($.55*(45*\i:3) + .55*(45*\i+45:3)$);
    }
    \draw (a0)--(a1)--(a2)--(a3)--(a4)--(a5)--(a6)--(a7)--cycle;

    \foreach \i in {0, 1, 2, 3} { 
        \tikzmath{\n = int(\i + 1);}
        \node at ($1.15*(l\i)$){$\n$};
    }
    \foreach \i in {4, 5, 6, 7} {
        \tikzmath{\n = int(mod(\i, 4) + 1);}
        \node at ($1.15*(l\i)$){$-\n$};
    }

    \draw (a4)--(a2)--(a5);
    \draw (a1)--(a6)--(a0);

    \draw[->] ($1.3*(a0)$)--($1.6*(a0)$);

    \coordinate (v1) at ($1/3*(a2) + 1/3*(a3) + 1/3*(a4)$);
    \coordinate (v2) at ($1/3*(a2) + 1/3*(a5) + 1/3*(a4)$);
    \coordinate (v3) at ($1/4*(a1) + 1/4*(a2) + 1/4*(a5) + 1/4*(a6)$);
    \coordinate (v4) at ($1/3*(a0) + 1/3*(a1) + 1/3*(a6)$);
    \coordinate (v5) at ($1/3*(a0) + 1/3*(a6) + 1/3*(a7)$);

    \color{red}
    \draw[thick] (v1)--(v2)--(v3)--(v4)--(v5);
    \draw[thick] (v1)--(l2);
    \draw[thick] (v1)--(l3);
    \draw[thick] (v2)--(l4);
    \draw[thick] (v3)--(l1);
    \draw[thick] (v3)--(l5);
    \draw[thick] (v4)--(l0);
    \draw[thick] (v5)--(l6);
    \draw[thick] (v5)--(l7);
\end{tikzpicture}

%% file: figures/centralToAxial2.tex
\begin{tikzpicture}[scale=1/2]
    \foreach \i in {0, 1, 2, 3, 4, 5, 6, 7} {
        \coordinate (a\i) at (45*\i:3);
        \coordinate (l\i) at ($.55*(45*\i:3) + .55*(45*\i+45:3)$);
    }
    
    \draw (a1)--(a2)--(a3)--(a4)--(a5);
    
    \coordinate (shift) at ($.25*(a7)$);
    \draw ($(shift)+(a5)$)--($(shift)+(a6)$)--($(shift)+(a7)$)--($(shift)+(a0)$)--($(shift)+(a1)$);

    \foreach \i in {0, 1, 2, 3} { 
        \tikzmath{\n = int(\i + 1);}
        \ifnum \i > 0
        \node at ($1.15*(l\i)$){$\n$};
        \else
        \node at ($1.15*(l\i) + (shift)$){$\n$};
        \fi
    }
    \foreach \i in {4, 5, 6, 7} {
        \tikzmath{\n = int(mod(\i, 4) + 1);}
        \ifnum \i < 5
        \node at ($1.15*(l\i)$){$-\n$};
        \else
        \node at ($1.15*(l\i) + (shift)$){$-\n$};
        \fi
    }

    \draw (a4)--(a2)--(a5);
    \draw ($(shift)+(a1)$)--($(shift)+(a6)$)--($(shift)+(a0)$);

    \draw[dashed] ($1.15*(a5)$)--($1.15*(a1)$);
    \draw[dashed] ($1.15*(a5)+(shift)$)--($1.15*(a1)+(shift)$);
    \node[right] at ($1.15*(a1)$){$l$};

    \draw[->] ($1.4*(a0)$)--($1.7*(a0)$);
\end{tikzpicture}

%% file: figures/centralToAxial3.tex
\begin{tikzpicture}[scale=1/2]
    \foreach \i in {0, 1, 2, 3, 4, 5, 6, 7} {
        \coordinate (a\i) at (45*\i:3);
        \coordinate (l\i) at ($.55*(45*\i:3) + .55*(45*\i+45:3)$);
    }
    \draw (a0)--(a1)--(a2)--(a3)--(a4)--(a5)--(a6)--(a7)--cycle;

    \node at ($1.15*(l0)$){$-2$};
    \node at ($1.15*(l1)$){$2$};
    \node at ($1.15*(l2)$){$3$};
    \node at ($1.15*(l3)$){$4$};
    \node at ($1.15*(l4)$){$-1$};
    \node at ($1.15*(l5)$){$1$};
    \node at ($1.15*(l6)$){$-4$};
    \node at ($1.15*(l7)$){$-3$};

    \draw (a4)--(a2)--(a5);
    \draw (a5)--(a0)--(a6);

    \draw[dashed] ($1.15*(a5)$)--($1.15*(a1)$);
    \node[right] at ($1.15*(a1)$){$l$};

    \coordinate (v1) at ($1/3*(a2) + 1/3*(a3) + 1/3*(a4)$);
    \coordinate (v2) at ($1/3*(a2) + 1/3*(a5) + 1/3*(a4)$);
    \coordinate (v3) at ($1/4*(a1) + 1/4*(a2) + 1/4*(a5) + 1/4*(a6)$);
    \coordinate (v4) at ($1/3*(a0) + 1/3*(a5) + 1/3*(a6)$);
    \coordinate (v5) at ($1/3*(a0) + 1/3*(a6) + 1/3*(a7)$);

    \color{red}
    \draw[thick] (v1)--(v2)--(v3)--(v4)--(v5);
    \draw[thick] (v1)--(l2);
    \draw[thick] (v1)--(l3);
    \draw[thick] (v2)--(l4);
    \draw[thick] (v3)--(l1);
    \draw[thick] (v3)--(l0);
    \draw[thick] (v4)--(l5);
    \draw[thick] (v5)--(l6);
    \draw[thick] (v5)--(l7);
\end{tikzpicture}

%% file: figures/treePics.tex
\tikzset{
    pics/atree/.style args={#1/#2/#3/#4/#5/#6/#7}{
      code = {
        \foreach \i in {0, 1, 2, 3, 4, 5} {
            \coordinate (a\i) at (-60*\i:1);
            \coordinate (l\i) at ($.55*(-60*\i:1) + .55*(-60*\i:1)$);
            
        }

        \node at ($1.25*(l1)$){\footnotesize$#3$};
        \node at ($1.25*(l2)$){\footnotesize$#4$};
        \node at ($1.25*(l3)$){\footnotesize$#5$};
        \node at ($1.25*(l4)$){\footnotesize$#6$};
        \node at ($1.25*(l5)$){\footnotesize$#1$};
        \node at ($1.25*(l0)$){\footnotesize$#2$};

        \coordinate (v1) at ($.35*(l0)$);
        \coordinate (v2) at ($-1*(v1)$);
        \draw (v1)--(v2);
        \draw (l5)--(v1)--(l0);
        \draw (l1)--(v1);
        \draw (l2)--(v2)--(l3);
        \draw (l4)--(v2);
        
        \node[rectangle, draw] at ($(135:1.9) + (-.1,0)$){$#7$};
    }}}

\tikzset{
    pics/btree/.style args={#1/#2/#3/#4/#5/#6/#7}{
      code = {
        \foreach \i in {0, 1, 2, 3, 4, 5} {
            \coordinate (a\i) at (-60*\i+120:1);
            \coordinate (l\i) at ($.55*(-60*\i + 120:1) + .55*(-60*\i + 60:1)$);
        }

        \node at ($1.25*(a1)$){\footnotesize$#3$};
        \node at ($1.25*(a2)$){\footnotesize$#4$};
        \node at ($1.35*(a3)$){\footnotesize$#5$};
        \node at ($1.35*(a4)$){\footnotesize$#6$};
        \node at ($1.25*(a5)$){\footnotesize$#1$};
        \node at ($1.25*(a0)$){\footnotesize$#2$};

        \coordinate (v1) at ($.5*(l0)$);
        \coordinate (v2) at ($0*(v1)$);
        \draw (v1)--(a0);
        \draw (v1)--(a1);
        \draw (v1)--(v2);
        \draw (v2)--(a2);
        \draw (v2)--(a3);
        \draw (v2)--(a4);
        \draw (v2)--(a5);

        \node[rectangle, draw] at ($(135:1.9) + (-.1,-.2)$){$#7$};
    }}}

\tikzset{
    pics/ctree/.style args={#1/#2/#3/#4/#5/#6/#7}{
      code = {
        \foreach \i in {0, 1, 2, 3, 4, 5, 6, 7, 8, 9, 10, 11} {
            \coordinate (a\i) at (-30*\i+30:1);
            \coordinate (l\i) at ($.55*(-30*\i + 30:1) + .55*(-30*\i - 30:1)$);
        }

        \node at ($1.25*(a2)$){\footnotesize$#3$};
        \node at ($1.25*(a9)$){\footnotesize$#4$};
        \node at ($1.35*(a6)$){\footnotesize$#5$};
        \node at ($1.35*(a8)$){\footnotesize$#6$};
        \node at ($1.25*(a11)$){\footnotesize$#1$};
        \node at ($1.25*(a0)$){\footnotesize$#2$};

        \coordinate (v1) at ($.65*(l0)$);
        \coordinate (v2) at ($-1*(v1)$);
        \coordinate (v3) at (0,0);
        \draw (v1)--(v2);
        \draw (v1)--(a0);
        \draw (v1)--(a2);
        \draw (v3)--(a11);
        \draw (v2)--(a6);
        \draw (v2)--(a8);
        \draw (v3)--(a9);

        \node[rectangle, draw] at ($(135:1.9) + (-.1,0)$){$#7$};
    }}}

\tikzset{
    pics/dtree/.style args={#1/#2/#3/#4/#5/#6}{
      code = {
        \foreach \i in {0, 1, 2, 3, 4, 5} {
            \coordinate (x\i) at (-60*\i+120:1);
            \coordinate (y\i) at ($.55*(-60*\i + 120:1) + .55*(-60*\i + 60:1)$);
        }

        \node at ($1.25*(x1)$){\footnotesize$#3$};
        \node at ($(0,.1)+1.25*(x2)$){\footnotesize$#4$};
        \node at ($1.35*(x3)$){\footnotesize$#5$};
        \node at ($1.35*(x4)$){\footnotesize$#6$};
        \node at ($(0,.1)+1.25*(x5)$){\footnotesize$#1$};
        \node at ($1.25*(x0)$){\footnotesize$#2$};

        \coordinate (v1) at ($.35*(90:1)$);
        \coordinate (v2) at ($-.35*(v1)$);
        \draw (v1)--(x0);
        \draw (v1)--(x1);
        \draw (v1)--(v2);
        \draw (0,.1)--($(x2)+(0,.1)$);
        \draw (v2)--(x3);
        \draw (v2)--(x4);
        \draw (0,.1)--($(x5)+(0,.1)$);
    }}}

\tikzset{
    pics/etree/.style args={#1/#2/#3/#4/#5/#6}{
      code = {
        \foreach \i in {0, 1, 2, 3, 4, 5, 6, 7, 8, 9, 10, 11} {
            \coordinate (x\i) at (-30*\i+30:1);
            \coordinate (y\i) at ($.55*(-30*\i + 30:1) + .55*(-30*\i - 30:1)$);
        }

        \coordinate (v1) at ($.65*(y0)$);
        \coordinate (v2) at ($-1*(v1)$);
        \coordinate (v3) at ($.5*(v1)$);
        \coordinate (v4) at ($-.5*(v1)$);
        \draw (v1)--(v2);
        \draw (v1)--(x0);
        \draw (v1)--(x2);
        \draw (v3)--($(v3) + (0,.75)$);
        \draw (v2)--(x6);
        \draw (v2)--(x8);
        \draw (v4)--($(v4) + (0,.75)$);

        \node at ($1.25*(x2)$){\footnotesize$#3$};
        \node at ($1.25*(v4) + 1.25*(0,.75)$){\footnotesize$#4$};
        \node at ($1.35*(x6)$){\footnotesize$#5$};
        \node at ($1.35*(x8)$){\footnotesize$#6$};
        \node at ($1.25*(v3) + 1.25*(0,.75)$){\footnotesize$#1$};
        \node at ($1.25*(x0)$){\footnotesize$#2$};
    }}}

\tikzset{
    pics/ftree/.style args={#1/#2/#3/#4/#5/#6}{
      code = {
        \foreach \i in {0, 1, 2, 3, 4, 5, 6, 7, 8, 9, 10, 11} {
            \coordinate (x\i) at (-30*\i+30:1);
            \coordinate (y\i) at ($.55*(-30*\i + 30:1) + .55*(-30*\i - 30:1)$);
        }

        \coordinate (v1) at (90:.65);
        \coordinate (v2) at (-30:.65);
        \coordinate (v3) at (210:.65);
        \coordinate (x9) at ($(v1) + (135:.35)$);
        \coordinate (x11) at ($(v1) + (45:.35)$);
        \coordinate (x2) at ($(v2) + (0:.35)$);
        \coordinate (x3) at ($(v2) + (-90:.35)$);
        \coordinate (x5) at ($(v3) + (-90:.35)$);
        \coordinate (x6) at ($(v3) + (180:.35)$);
        
        \draw (v1)--(x9);
        \draw (v1)--(x11);
        \draw (v1)--(0,0);
        \draw (v2)--(x2);
        \draw (v2)--(x3);
        \draw (v2)--(0,0);
        \draw (v3)--(x5);
        \draw (v3)--(x6);
        \draw (v3)--(0,0);

        \node at ($1.35*(x9)$){\footnotesize$#1$};
        \node at ($1.35*(x11)$){\footnotesize$#2$};
        \node at ($1.35*(x2)$){\footnotesize$#3$};
        \node at ($1.35*(x3)$){\footnotesize$#4$};
        \node at ($1.35*(x5)$){\footnotesize$#5$};
        \node at ($1.35*(x6)$){\footnotesize$#6$};
    }}}

%% file: figures/C2graph.tex
\tikzset{hexagon/.pic={
    
}}

\tikzset{
    pics/chexagon/.style args={#1/#2/#3/#4/#5/#6/#7}{
      code = {
        \foreach \i in {0, 1, 2, 3, 4, 5} {
            \coordinate (k\i) at (-60*\i+120:1);
            \coordinate (p\i) at ($.6*(-60*\i+120:1) + .6*(-60*\i+180:1)$);
            
        }

        \node at ($1.15*(p1)$){\scriptsize$#1$};
        \node at ($1.15*(p2)$){\scriptsize$#2$};
        \node at ($1.15*(p3)$){\scriptsize$#3$};
        \node at ($1.15*(p4)$){\scriptsize$-#1$};
        \node at ($1.15*(p5)$){\scriptsize$-#2$};
        \node at ($1.15*(p0)$){\scriptsize$-#3$};

        \draw (k0)--(k1)--(k2)--(k3)--(k4)--(k5)--cycle;

        \draw (k#4)--(k#5);
        \tikzmath{\n = int(mod(#4 + 3, 6)); \m = int(mod(#5 + 3, 6));}
        \draw (k\n)--(k\m);

        \draw (k#6)--(k#7);
        \tikzmath{\n = int(mod(#6 + 3, 6)); \m = int(mod(#7 + 3, 6));}
        \draw (k\n)--(k\m);
    }}}

\tikzset{
    pics/ahexagon/.style args={#1/#2/#3/#4/#5/#6/#7}{
      code = {
        \foreach \i in {0, 1, 2, 3, 4, 5} {
            \coordinate (h\i) at (-60*\i+120:1);
            \coordinate (j\i) at ($.6*(-60*\i+120:1) + .6*(-60*\i+180:1)$);
            
        }

        \node at ($1.15*(j1)$){\scriptsize$#1$};
        \node at ($1.15*(j2)$){\scriptsize$#2$};
        \node at ($1.15*(j3)$){\scriptsize$#3$};
        \node at ($1.15*(j4)$){\scriptsize$-#3$};
        \node at ($1.15*(j5)$){\scriptsize$-#2$};
        \node at ($1.15*(j0)$){\scriptsize$-#1$};

        \draw (h0)--(h1)--(h2)--(h3)--(h4)--(h5)--cycle;

        \draw (h#4)--(h#5);
        \tikzmath{\n = int(mod(#4, 6)); \m = int(mod(-(#5 - 4) + 4, 6));}
        \draw (h\n)--(h\m);

        \draw (h#6)--(h#7);
        \tikzmath{\n = int(mod(#6, 6)); \m = int(mod(-(#5 - 4) + 4, 6));}
        \draw (h\n)--(h\m);
    }}}

\tikzset{
    pics/ctree/.style args={#1/#2/#3/#4/#5/#6/#7/#8/#9}{
      code = {
        \foreach \i in {0, 1, 2, 3, 4, 5} {
            \coordinate (x\i) at (-60*\i+120:1);
            \coordinate (e\i) at ($.55*(-60*\i+120:1) + .55*(-60*\i+180:1)$);
            
        }

        \node at ($1.15*(e1)$){\scriptsize$#1$};
        \node at ($1.15*(e2)$){\scriptsize$#2$};
        \node at ($1.15*(e3)$){\scriptsize$#3$};
        \node at ($1.15*(e4)$){\scriptsize$-#1$};
        \node at ($1.15*(e5)$){\scriptsize$-#2$};
        \node at ($1.15*(e0)$){\scriptsize$-#3$};

        \coordinate (v1) at ($.4*(e#4) + .4*(e#6)$);
        \coordinate (v2) at ($.4*(e#7) + .4*(e#9)$);
        \draw (v1)--(v2);
        \draw ($.9*(e#4)$)--(v1)--($.9*(e#6)$);
        \draw (v1)--($.9*(e#5)$);
        \draw ($.9*(e#7)$)--(v2)--($.9*(e#9)$);
        \draw (v1)--($.9*(e#8)$);
    }}}
\scalebox{.825}
{\begin{tikzpicture}[scale=2]
    \foreach \i in {0, 1, 2, 3, 4, 5} {
        \coordinate (a\i) at (60*\i:2.5);
        \coordinate (b\i) at (60*\i:1.75);
        \coordinate (c\i) at (60*\i+30:.75);
        
        \coordinate (l\i) at ($.55*(60*\i:3) + .55*(60*\i+60:3)$);
    }
    \coordinate (alpha) at (0,0);

    \draw[ultra thick] (a0)--(a1)--(a2)--(a3)--(a4)--(a5)--cycle;
    \draw (b1)--(b3)--(b5)--cycle;
    \draw[ultra thick] (a2)--(c1);
    \draw (c1)--(b5);
    \draw[ultra thick] (a4)--(c3);
    \draw (c3)--(b1);
    \draw[ultra thick] (a0)--(c5);
    \draw (c5)--(b3);
    \draw[ultra thick] (c1)--(alpha);
    \draw[ultra thick] (c3)--(alpha);
    \draw[ultra thick] (c5)--(alpha);
    \draw (a1)--(b1) (a3)--(b3) (a5)--(b5);

    \node[below] at (alpha){$13$};

    \node[left] at (c1){$10$};
    \node[left] at (c3){$11$};
    \node[above] at (c5){$12$};

    \node[above left] at (b1){$7$};
    \node[above left] at (b3){$8$};
    \node[below left] at (b5){$9$};

    \node at ($1.1*(a1)$){$4$};
    \node at ($1.1*(a2)$){$1$};
    \node at ($1.1*(a5)$){$6$};
    \node at ($1.1*(a0)$){$3$};
    \node at ($1.1*(a3)$){$5$};
    \node at ($1.1*(a4)$){$2$};


    \draw[ultra thick, red] (a5)--(a0)--(a1);
    \draw[red] (a1)--(b1)--(b5)--(a5);

    \draw[ultra thick, blue] (alpha)--(c3)--(a4)--(a3)--(a2)--(c1)--cycle;
    
    \foreach \i in {0,1,2,3,4,5} {
        \filldraw (a\i) circle (1.5pt);

        \tikzmath{\n = int(mod(\i, 2));}
        \ifnum \n = 1
        \filldraw (b\i) circle (1pt);
        \filldraw (c\i) circle (1.5pt);
        \fi
    }
    \filldraw (alpha) circle (1.5pt);
    
    \filldraw[red] (b1) circle (1pt);
    \filldraw[red] (b5) circle (1pt);
    \filldraw[red] (a0) circle (1.5pt);
    \filldraw[red] (a1) circle (1.5pt);
    \filldraw[red] (a5) circle (1.5pt);

    \filldraw[blue] (a2) circle (1.5pt);
    \filldraw[blue] (a3) circle (1.5pt);
    \filldraw[blue] (a4) circle (1.5pt);
    \filldraw[blue] (c1) circle (1.5pt);
    \filldraw[blue] (c3) circle (1.5pt);
    \filldraw[blue] (alpha) circle (1.5pt);

    \coordinate (T1) at ($(b1)!.35!(b5)$);
    \pic at ($(T1) + (.7,0)$) [scale=.75] {dtree=-2/-1/1/2/3/-3};
    \draw[->] ($(T1) + (.25,-.2)$)--($(T1) + (.05,-.2)$);

    \coordinate (T2) at ($(a2)!.5!(a3)$);
    \pic at ($(T2) + (-.75,0)$) [scale=.75] {etree=-2/-3/1/2/-1/3};
    \draw[->] ($(T2) + (-.4,0)$)--($(T2) + (-.1,0)$);

    \coordinate (T3) at ($(a4)!.4!(a5)$);
    \pic at ($(T3) + (0,.5)$) [scale=.75] {ftree=-3/3/1/2/-2/-1};
    \draw[->] ($(T3) + (.5,.15)$)--($(T3) + (1.35,.15)$);

    \coordinate (T4) at ($(a0)!.5!(a1)$);
    \pic at ($(T4) + (.75,0)$) [scale=.75, white] {etree=3/-3/-1/2/-2/1};
\end{tikzpicture}}

%% file: paper.bbl
\begin{thebibliography}{10}
\setlength{\itemsep}{-0.1mm}

\bibitem{ABHY} N.~Arkani-Hamed, Y.~Bai, S.~He, G.~Yan:
{\em Scattering Forms and the Positive Geometry of Kinematics, Color and the Worldsheet},
Journal of High Energy Physics \textbf{05} (2018) 096.

\bibitem{AHL} N.~Arkani-Hamed, S.~He, T.~Lam:
{\em Cluster Configuration Spaces of Finite Type},
SIGMA \textbf{17} (2021) 092.

\bibitem{AHLT} N.~Arkani-Hamed, S.~He, T.~Lam, H.~Thomas:
{\em Binary geometries, generalized particles and strings, and cluster algebras}, Physical Review D \textbf{107(6)} (2023) 066015.

\bibitem{BDMTY} V.~Bazier-Matte, G.~Douville, K.~Mousavand, H.~Thomas, E.~Yildirim:
{\em ABHY Associahedra and Newton polytopes of $F$-polynomials for finite type cluster algebras}, 
Journal of London Mathematical Society \textbf{109} (2024) e12817.


\bibitem{BHV} L.~J.~Billera, S.~P.~Holmes, K.~Vogtmann:
{\em Geometry of the Space of Phylogenetic Trees},
Advances in Applied Mathematics \textbf{27} (2001) 733--767.

\bibitem{Bro} F.~Brown: 
{\em Multiple zeta values and periods of moduli spaces $\cM_{0,n}$}, Annales Sci. Ecole Norm. Sup. (2009) 42--371.

\bibitem{CFZ} F.~Chapoton, S.~Fomin, A.~Zelevinsky:
{\em Polytopal Realizations of Generalized Associahedra},
Canadian Mathematical Bulletin \textbf{45} (2002) 537--566.

\bibitem{FZ} S.~Fomin, A.~Zelevinzky:
{\em $Y$-systems and generalized associahedra},
Annals of Mathematics \textbf{158} (2003) 977--1018.

\bibitem{GKM} A.~Gathmann, M.~Kerber, H.~Markwig: 
{\em Tropical fans and the moduli spaces of tropical curves}, Compositio Mathematica. \textbf{145} (2009) 173--195.

\bibitem{M2}
 D.~Grayson and M.~Stillman:
{\em Macaulay2, a software system for research in algebraic geometry},
          available at \href{http://www.math.uiuc.ed/Macaulay2/}{\texttt{http://www.math.uiuc.ed/Macaulay2/}}

\bibitem{HLPZ} S.~He, Z.~Li, P.~Raman, C.~Zhang:
{\em Stringy canonical forms and binary geometries from associahedra, cyclohedra and generalized permutohedra},
Journal of High Energy Physics, \textbf{2020} (2020) 54.


\bibitem{L} T.~Lam:
{\em Moduli spaces in positive geometry},
2024, \href{https://arxiv.org/abs/2405.17332}{arXiv:2405.17332}.

\bibitem{ITG} D.~Maclagan, B.~Sturmfels:
{\em Introduction to Tropical Geometry}, American Mathematical Society, Graduate Studies in Mathematics \textbf{161} (2015).

\bibitem{RW} A.~Robinson, S.~Whitehouse:
{\em The tree representation of $\Sigma_{n+1}$}, Elsevier, Journal of Pure and Applied Algebra \textbf{111(1--3)} (1996) 245--253.

\bibitem{SpSt} D.~Speyer and B.~Sturmfels:
{\em The tropical Grassmannian}, Advances in Geometry \textbf{4} (2004) 389--411.

\bibitem{SpW} D.~Speyer, L.~Williams:
{\em The Tropical Totally Positive Grassmannian}, Journal of Algebraic Combinatorics \textbf{22} (2005) 189--210.

\bibitem{T} J.~Tevelev, 
{\em Compactifications of Subvarieties of Tori}, 
American Journal of Mathematics \textbf{129} (2007) 1087--1104.

\end{thebibliography}
